\documentclass[11pt,epsf]{article}
\usepackage{graphicx}
\usepackage{amsthm}
\usepackage{amsfonts}
\usepackage{amssymb}
\title{On a partial theta function and its spectrum}
\author{Vladimir Petrov Kostov\\ 
Universit\'e de Nice, 
Laboratoire de Math\'ematiques, Parc Valrose,\\ 06108 Nice Cedex 2, France,  
e-mail: kostov@math.unice.fr} 
\date{}
\newtheorem{tm}{Theorem}
\newtheorem{defi}[tm]{Definition}
\newtheorem{rem}[tm]{Remark}
\newtheorem{rems}[tm]{Remarks}
\newtheorem{lm}[tm]{Lemma}

\newtheorem{prop}[tm]{Proposition}
\newtheorem{nota}[tm]{Notation}
\begin{document} 
\maketitle 
\begin{abstract}
The bivariate series $\theta (q,x):=\sum _{j=0}^{\infty}q^{j(j+1)/2}x^j$ 
defines  
a {\em partial theta function}. For fixed $q$ ($|q|<1$), $\theta (q,.)$ is an 
entire function. For $q\in (-1,0)$ the function $\theta (q,.)$ 
has infinitely many negative and infinitely many positive real zeros.
There exists 
a sequence $\{ \bar{q}_j\}$ of values of $q$  
tending to $-1^+$ such that $\theta (\bar{q}_k,.)$ has a double real zero 
$\bar{y}_k$ (the rest of its real zeros being simple). For $k$ odd 
(resp. for $k$ even) 
$\theta (\bar{q}_k,.)$ has a local minimum at $\bar{y}_k$ and $\bar{y}_k$ 
is the rightmost of the real negative zeros of $\theta (\bar{q}_k,.)$ 
(resp.  
$\theta (\bar{q}_k,.)$ has a local maximum at $\bar{y}_k$ and 
for $k$ sufficiently large $\bar{y}_k$ 
is the second from the left 
of the real negative zeros of $\theta (\bar{q}_k,.)$).
For $k$ sufficiently large one has $-1<\bar{q}_{k+1}<\bar{q}_k<0$. 
One has $\bar{q}_k=1-(\pi /8k)+o(1/k)$ and 
$|\bar{y}_k|\rightarrow e^{\pi /2}=4.810477382\ldots$.
\\ 


\end{abstract}

\section{Introduction}

The bivariate series $\theta (q,x):=\sum _{j=0}^{\infty}q^{j(j+1)/2}x^j$ defines 
an entire function in $x$ for every fixed $q$ from the open unit disk. 
This function is called a {\em partial theta function} because 
$\theta (q^2,x/q)=\sum _{j=0}^{\infty}q^{j^2}x^j$ whereas the Jacobi theta function 
is defined by the same series, but when summation is performed from 
$-\infty$ to $\infty$ (i.e. when summation is not partial). 

There are several domains in which the function $\theta$ finds applications: 
in the theory 
of (mock) modular forms (see \cite{BrFoRh}), in statistical physics 
and combinatorics (see \cite{So}), in asymptotic analysis (see \cite{BeKi}) 
and in  Ramanujan type $q$-series 
(see \cite{Wa}). Recently it has been considered in the context of problems 
about hyperbolic polynomials 
(i.e. real polynomials having all their zeros real, see 
\cite{Ha}, \cite{Pe}, \cite{Hu}, \cite{Ost}, \cite{KaLoVi}, 
\cite{KoSh} and \cite{Ko2}). These 
problems have been studied earlier by Hardy, Petrovitch and Hutchinson 
(see \cite{Ha}, \cite{Hu} and \cite{Pe}). For more information about $\theta$,  
see also~\cite{AnBe}.

For $q\in \mathbb{C}$, $|q|\leq 0.108$, the function $\theta (q,.)$ 
has no multiple zeros, see \cite{Ko4}. For $q\in [0,1)$ the function 
$\theta$ has been studied in \cite{KoSh}, \cite{Ko1}, 
\cite{Ko2} and~\cite{Ko3}. The results are summarized in the following 
theorem:

\begin{tm}\label{known}
(1) For any $q\in (0,1)$ the function $\theta (q,.)$ has infinitely many 
negative zeros.

(2) There exists 
a sequence of values of $q$ (denoted by $0<\tilde{q}_1<\tilde{q}_2<\cdots$) 
tending to $1^-$ such that $\theta (\tilde{q}_k,.)$ has a double negative zero 
$y_k$ which is the rightmost of its real zeros and which is a local minimum of 
$\theta (q,.)$. One has $\tilde{q}_1=0.3092493386\ldots$.

(3) For the remaining values of $q\in [0,1)$ the function $\theta (q,.)$ 
has no multiple real zero. 

(4) For $q\in (\tilde{q}_k, \tilde{q}_{k+1})$ (we set $\tilde{q}_0=0$) the 
function $\theta (q,.)$ has exactly $k$ complex conjugate pairs of zeros 
counted with multiplicity.

(5) One has $\tilde{q}_k=1-(\pi /2k)+o(1/k)$ and 
$y_k\rightarrow -e^{\pi}=-23.1407\ldots$.
\end{tm}

\begin{defi}
{\rm A value of $q\in \mathbb{C}$, $|q|<1$, is said to belong to the 
{\em spectrum} of $\theta$ if $\theta (q,.)$ has a multiple zero.}
\end{defi}

In the present paper we consider the function $\theta$ in the case when 
$q\in (-1,0]$. In order to use the results about the case $q\in [0,1)$ one 
can notice the following fact. For $q\in (-1,0]$ set $v:=-q$. Then 

\begin{equation}\label{uq}
\theta (q,x)=\theta (-v,x)=\theta (v^4,-x^2/v)-vx\theta (v^4,-vx^2)~.
\end{equation}
We prove the analog of the above theorem. The following three theorems are 
proved respectively in Sections~\ref{secnewfirst}, \ref{secnew} 
and~\ref{secastm}.

\begin{tm}\label{newfirst}
For any $q\in (-1,0)$ the function $\theta (q,.)$ 
has infinitely many negative and infinitely many positive real zeros.
\end{tm}

\begin{tm}\label{new}
(1) There exists 
a sequence of values of $q$ (denoted by $\bar{q}_j$) 
tending to $-1^+$ such that $\theta (\bar{q}_k,.)$ has a double real zero 
$\bar{y}_k$ (the rest of its real zeros being simple). 
For the remaining values of $q\in (-1,0)$ the function $\theta (q,.)$ 
has no multiple real zero.
 
(2) For $k$ odd (resp. for $k$ even) one has $\bar{y}_k<0$, 
$\theta (\bar{q}_k,.)$ has a local minimum at $\bar{y}_k$ and $\bar{y}_k$ 
is the rightmost of the real negative zeros of $\theta (\bar{q}_k,.)$ 
(resp. $\bar{y}_k>0$, 
$\theta (\bar{q}_k,.)$ has a local maximum at $\bar{y}_k$ and 
for $k$ sufficiently large $\bar{y}_k$ 
is the leftmost but one (second from the left) 
of the real negative zeros of $\theta (\bar{q}_k,.)$).
 
(3) For $k$ sufficiently large one has $-1<\bar{q}_{k+1}<\bar{q}_k<0$. 

(4) For $k$ sufficiently large and for 
$q\in (\bar{q}_{k+1}, \bar{q}_k)$ 
the function $\theta (q,.)$ has exactly $k$ complex conjugate pairs of zeros 
counted with multiplicity.
\end{tm}

\begin{rem}
{\rm Numerical experience confirms the conjecture that 
parts (2), (3) and (4) of the theorem are true for any $k\in \mathbb{N}$. 
Proposition~\ref{firstzero} clarifies part (2) of the theorem.}
\end{rem}

\begin{tm}\label{astm}
One has $\bar{q}_k=1-(\pi /8k)+o(1/k)$ and 
$|\bar{y}_k|\rightarrow e^{\pi /2}=4.810477382\ldots$.
\end{tm}

\begin{rems}\label{remsKa}
{\rm (1) Theorems~\ref{known} and \ref{new} do not tell whether 
there are values of $q\in (-1,1)$ 
for which $\theta (q,.)$ has a multiple complex conjugate pair of zeros.

(2) It would be interesting to know whether the sequences $\{ y_k\}$ and 
$\{ \bar{y}_{2k-1}\}$ are monotone decreasing and $\{ \bar{y}_{2k}\}$ 
is monotone increasing. This is true for at least the five first terms of each 
sequence.

(3) It would be interesting to know whether there are complex non-real 
values of $q$ of the open unit disk belonging to the spectrum of $\theta$ 
and (as suggested by A.~Sokal) whether $|\tilde{q}_1|$ 
is the smallest of the moduli of the spectral values.

(4) The following statement is formulated and proved in \cite{Ka}: 

{\em The sum of the series 
$\sum _{j=0}^{\infty}q^{j(j+1)/2}x^j$ (considered for $q\in (0,1)$ and 
$x\in \mathbb{C}$) tends to $1/(1-x)$ 
(for $x$ fixed and as $q\rightarrow 1^-$)  
exactly when 
$x$ belongs to the interior of the closed Jordan curve 
$\{ e^{|s|+is}, s\in [-\pi ,\pi]\}$.} 

This statement and equation (\ref{uq}) 
imply that as $q\rightarrow -1^+$, $\theta (q,x)\rightarrow (1-x)/(1+x^2)$ 
for $x\in (-e^{\pi /2},e^{\pi /2})$, $e^{\pi /2}=4.810477381\ldots$. 
Notice that the radius of convergence of 
the Taylor series at $0$ of the function $(1-x)/(1+x^2)$ equals 1.

(5) On Fig.~\ref{parthetanegfirstfour} and \ref{parthetanegsecondfour} we show 
the graphs of $\theta (\bar{q}_k,.)$ for $k=1$, $\ldots$, $8$. The ones for 
$k=1$, $2$, $5$ and $6$ are shown in black, the others are drawn in grey. One 
can notice by looking at Fig.~\ref{parthetanegsecondfour} 
that for $x\in [-2.5,2.5]$ the graphs of $\theta (\bar{q}_k,.)$ for $k\geq 5$ 
are hardly distinguishable from the one of $(1-x)/(1+x^2)$.

(6) The approximative values of $\bar{q}_k$ and $\bar{y}_k$ 
for $k=1$, $2$, $\ldots$ 
$8$ are:}

$$\begin{array}{rcccccccc}k&~~~~~&1&~~&2&~~&3&~~&4\\ 
-\bar{q}_k&~~~~~&0.72713332&~~&0.78374209&~~&
0.84160192&~~&0.86125727\\ 
\bar{y}_k&~~~~~&-2.991&~~&2.907&~~&-3.621&~~&3.523\\ \\ 
k&~~~~~&5&~~&6&~~&7&~~&8\\
-\bar{q}_k&~~~~~&0.88795282&~~&0.89790438&~~&
0.913191~~&~~&0.9192012~\\ 
\bar{y}_k&~~~~~&-3.908&~~&3.823&~~&-4.08~&~~&4.002
\end{array}$$

\end{rems} 

\begin{figure}[htbp]
\centerline{\hbox{\includegraphics[scale=0.7]{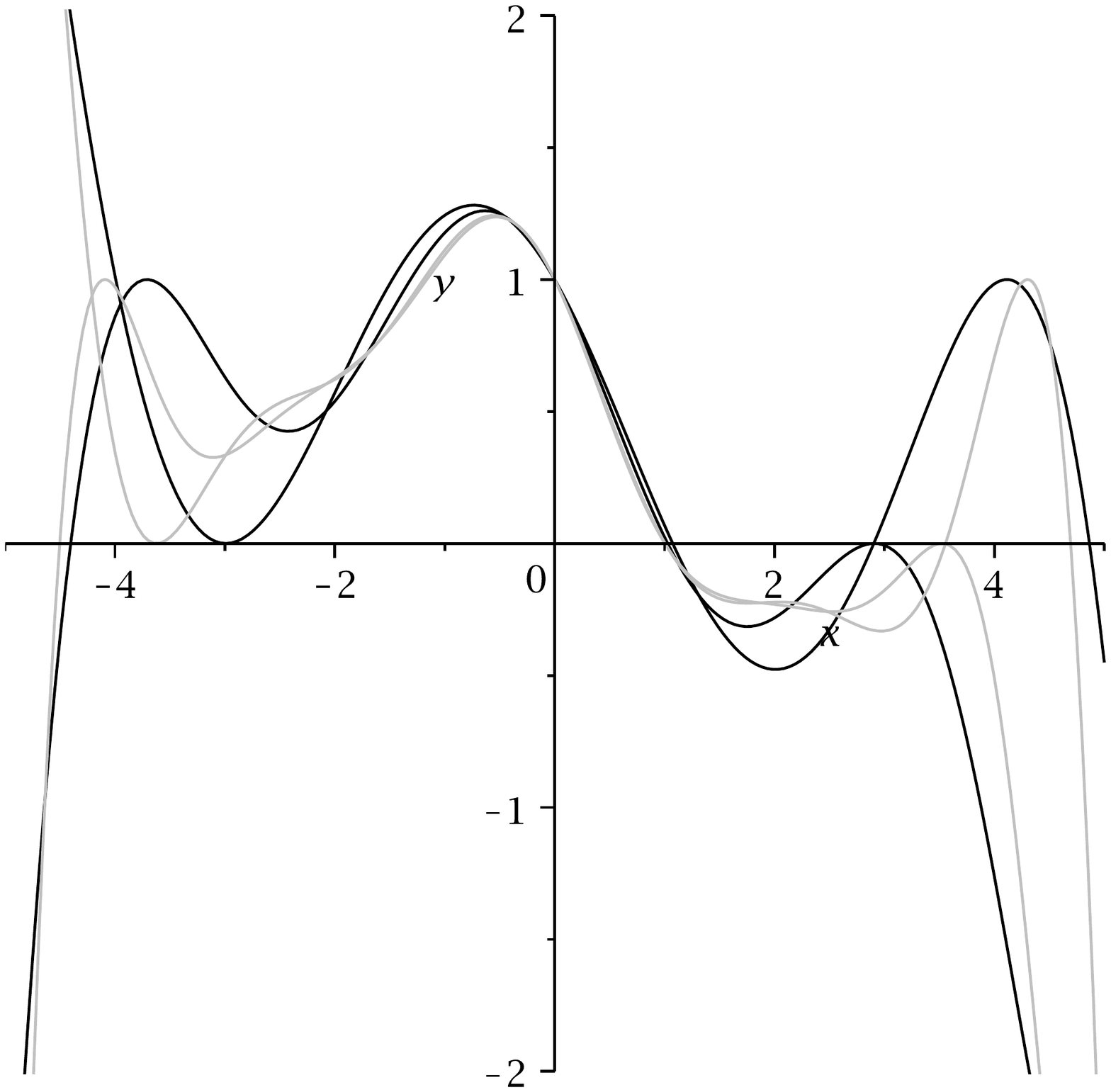}}}
    \caption{The graphs of the functions 
$\theta (\bar{q}_k,.)$ for $k=1$, $2$, $3$ and $4$.}
\label{parthetanegfirstfour}
\end{figure}

\begin{figure}[htbp]
\centerline{\hbox{\includegraphics[scale=0.7]{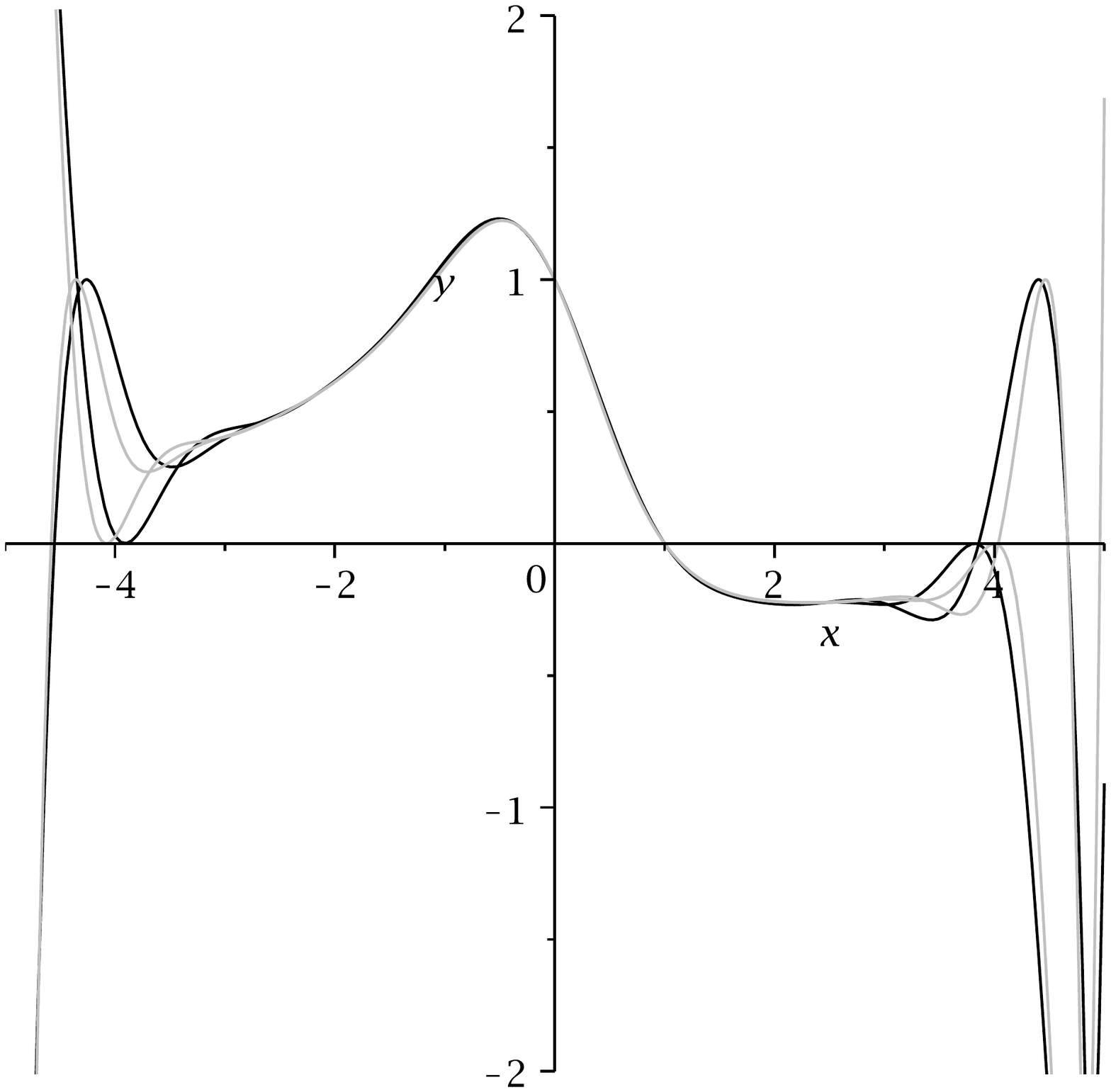}}}
    \caption{The graphs of the functions 
$\theta (\bar{q}_k,.)$ for $k=5$, $6$, $7$ and $8$.}
\label{parthetanegsecondfour}
\end{figure}

\section{Some facts about $\theta$}

This section contains properties of the function $\theta$, known or 
proved in \cite{Ko2}. When a property is valid for all $q$ from the unit disk 
or for all $q\in (-1,1)$, we write $\theta (q,x)$. When a property holds true 
only for $q\in [0,1)$ or only for $q\in (-1,0]$, we write $\theta (v,x)$ or 
$\theta (-v,x)$, where $v\in [0,1)$. 

\begin{tm}\label{known1}
(1) The function $\theta$ satisfies the following functional equation:

\begin{equation}\label{FE}
\theta (q,x)=1+qx\theta (q,qx)
\end{equation}
and the following differential equation:

\begin{equation}\label{DE}
2q\partial \theta /\partial q=2x\partial \theta /\partial x+
x^2\partial ^2\theta /\partial x^2=x\partial ^2(x\theta )/\partial x^2 
\end{equation}

(2) For $k\in \mathbb{N}$ one has $\theta (v,-v^{-k})\in (0,v^k)$.

(3) In the following two situations the two conditions 
sgn$(\theta (v,-v^{-k-1/2}))=(-1)^k$ 
and $|\theta (v,-v^{-k-1/2})|>1$ hold true:

~~~~~(i) For $k\in \mathbb{N}$ and $v>0$ small enough;

~~~~~(ii) For any $v\in (0,1)$ fixed and for $k\in \mathbb{N}$ large enough.
\end{tm}

The real entire
function $\psi(z)$ is said to belong to the Laguerre-P\'olya class
${\cal L-P}$ if it can be represented as
$$
\psi(x) = c x^{m} e^{-\alpha x^{2} + \beta x} \prod_{k=1}^{\omega}
(1+x/x_{k}) e^{-x/x_{k}},
$$
where $\omega$ is a natural number or infinity, $c$, $\beta$ and 
$x_{k}$ are real,
$\alpha \geq 0$, $m$ is a nonnegative integer and $\sum x_{k}^{-2} < \infty$.
Similarly, the
real entire function $\psi_*(x)$ is a function of type I in the
Laguerre-P\'{o}lya class, written $\psi_* \in {\cal L-PI}$, if $\psi_*(x)$
or $\psi_*(-x)$ can be represented in the form
\begin{equation}
\psi_*(x) = c x^{m} e^{\sigma x} \prod_{k=1}^{\omega} (1+x/x_{k}),
\label{1.2}
\end{equation}
where $c$ and $\sigma$ are real, $\sigma \geq 0$, $m$ is a nonnegative
integer, $x_{k}>0$, and $\sum 1/x_{k} < \infty$. It is clear that
${\cal L-PI}\subset{\cal L-P}$.  The functions in ${\cal L-P}$, 
and only these, are
uniform limits, on compact subsets of $\mathbb{C}$, of hyperbolic polynomials
(see, for example, Levin \cite[Chapter 8]{Le}). Similarly, 
$\psi \in {\cal L-PI}$
if and only if $\psi$ is a uniform limit on the compact sets 
of the complex plane of
polynomials whose zeros are real and are either all positive, or all
negative. Thus, the classes ${\cal L-P}$ and ${\cal L-PI}$ are closed under
differentiation; that is, if $\psi \in {\cal L-P}$, then $\psi^{(\nu)} \in
{\cal L-P}$ for every $\nu \in \mathbb{N}$ and similarly, if $\psi \in
{\cal L-PI}$, then $\psi^{(\nu)} \in {\cal L-PI}$ . P\'{o}lya and Schur
\cite{PolSch14} proved that if
\begin{equation}
\psi(x) = \sum_{k=0}^{\infty} \gamma_{k} \frac{x^{k}}{k!}
\label{Maclaurin}
\end{equation}
belongs to ${\cal L-P}$ and its Maclaurin coefficients $\gamma_k=\psi^{(k)}(0)$
are all nonnegative, then $\psi \in {\cal L-PI}$.

The following theorem is the basic result contained in~\cite{Ko5}:

\begin{tm}\label{fundtm}
(1) For any fixed $q\in \mathbb{C}^*$, $|q|<1$, and for $k$ sufficiently 
large, the function $\theta (q,.)$ has a zero $\zeta _k$ close to $-q^{-k}$ 
(in the sense that $|\zeta _k+q^{-k}|\rightarrow 0$ as $k\rightarrow \infty$). 
These are all but finitely-many of the zeros of $\theta$.

(2) For any $q\in \mathbb{C}^*$, $|q|<1$, one has 
$\theta (q,x)=\prod _k(1+x/x_k)$, where $-x_k$ are the zeros of $\theta$ 
counted with multiplicity.

(3) For $q\in (\tilde{q}_{j},\tilde{q}_{j+1}]$ the function 
$\theta (q,.)$ is a product of a degree $2j$ real polynomial without real roots 
and a function of the Laguerre-P\'olya class $\cal{L-PI}$. 
Their respective forms are $\prod _{k=1}^{2j}(1+x/\eta _k)$ and 
$\prod _k(1+x/\xi _k)$, where $-\eta _k$ and $-\xi _k$ are  
the complex and the real zeros of $\theta$ counted with multiplicity.

(4) For any fixed $q\in \mathbb{C}^*$, $|q|<1$, the function $\theta (q,.)$ has 
at most finitely-many multiple zeros.

(5) For any $q\in (-1,0)$ the function $\theta (q,.)$ is a product of the form 
$R(q,.)\Lambda (q,.)$, where $R=\prod _{k=1}^{2j}(1+x/\tilde{\eta}_k)$ 
is a real polynomial with constant term $1$ and without real zeros and 
$\Lambda =\prod _k(1+x/\tilde{\xi}_k)$, $\tilde{\xi}_k\in \mathbb{R}^*$, is a 
function of the Laguerre-P\'olya class $\cal{L-P}$. One has 
$\tilde{\xi}_k\tilde{\xi}_{k+1}<0$. The sequence $\{ |\tilde{\xi}_k|\}$ is 
monotone increasing for $k$ large enough.
\end{tm}

\section{Proof of Theorem~\protect\ref{newfirst}
\protect\label{secnewfirst}}

One can use equation (\ref{uq}). By part (3) of 
Theorem~\ref{known1} with $v^4$ for $v$, 
for each $v\in (0,1)$ fixed and for $k$ large enough, if 
$-x^2/v=-(v^4)^{-k-1/2}$ (i.e. if $|x|=v^{-2k-1/2}$), then  
$|\theta (v^4,-x^2/v)|>1$ and sgn$(\theta (v^4,-x^2/v))=(-1)^k$. 
At the same time part (2) of Theorem~\ref{known1} implies that for 
$-vx^2=-(v^4)^{-k}$ (i.e. again for $|x|=v^{-2k-1/2}$) one has 
$\theta (v^4,-vx^2)\in (0,v^{4k})$ hence 
$|vx\theta (v^4,-vx^2)|<v^{2k+1/2}<1$. This means that for 
$v\in (0,1)$ fixed and for 
$k$ large enough the equality sgn$(\theta (v^4,-x^2/v))=(-1)^k$ holds, i.e. 
there is a zero of $\theta$ on each interval of the form 
$(-v^{-2k-1/2},-v^{-2k+1/2})$ and $(v^{-2k+1/2},v^{-2k-1/2})$.

\section{Proof of Theorem~\protect\ref{new}
\protect\label{secnew}}

\subsection{Properties of the zeros of $\theta$}

The present subsection contains some preliminary information about the zeros 
of $\theta$. 

\begin{lm}\label{lm108}
For $q\in [-0.108, 0)$ all zeros of $\theta (q,.)$ are real and distinct.
\end{lm}

\begin{proof}
Indeed, it is shown in \cite{Ko4} that for $q\in \mathbb{C}$, 
$|q|\leq 0.108$, the zeros of $\theta$ are of the form $-q^{-j}\Delta _j$, 
$\Delta _j\in [0.2118, 1.7882]$. This implies (see \cite{Ko4}) 
that the moduli of all zeros are distinct for $|q|\leq 0.108$. When $q$ 
is real, then as all coefficients of $\theta$ are real, each of its 
zeros is either real or belongs to a complex conjugate pair. The moduli of 
the zeros being distinct the zeros are all real and distinct. 
\end{proof}

\begin{nota}
{\rm We denote by $0<x_1<x_3<\cdots$ the positive and by $\cdots <x_4<x_2<0$ 
the negative zeros of $\theta$. For $q\in [-0.108, 0)$ this notation 
is in line with the fact that $x_j$ is close to $-q^{-j}$.}
\end{nota}

\begin{rems}\label{zerosoftheta}
{\rm (1) The quantities $\Delta _j$ are constructed in \cite{Ko4} as 
convergent Taylor series in $q$ of the form $1+O(q)$.

(2) For $q\in (-1,0)$ the function $\theta (q,.)$ has no zeros in $[-1,0)$. 
Indeed, this follows from} 

$$\theta =1+q^3x^2+q^{10}x^4+\cdots +qx(1+q^5x^2+q^{14}x^4+\cdots )~.$$
{\rm Each of the two series is sign-alternating, 
with positive first term and 
with decreasing moduli of its terms for $q\in (-1,0)$, $x\in [-1,0)$. 
Hence their sums are positive; as $qx>0$, one has $\theta >0$.}
\end{rems}


\begin{lm}
For $q\in [-0.108, 0)$ the real zeros of $\theta$ and their products 
with $q$ are arranged on the real 
line as shown on Fig.~\ref{thetaqneg}.
\end{lm}

\begin{proof}
The lemma follows from equation (\ref{FE}). Indeed, 

$$0=\theta (q,x_{4k+2})=1+qx_{4k+2}\theta (q,qx_{4k+2})~~,~~
x_{4k+2}<0~~{\rm and}~~q<0~~{\rm hence}~~
\theta (q,qx_{4k+2})<0~.$$
For small values of $q$ the quantity $qx_{4k+2}$ is close to $-q^{-4k-1}$ (see 
Lemma~\ref{lm108} and part (1) of Remarks~\ref{zerosoftheta}), i.e. 
close to $x_{4k+1}$ and as $\theta (q,qx_{4k+2})<0$, one must have 
$x_{4k+1}<qx_{4k+2}<x_{4k+3}$. By continuity these inequalities hold for all $q<0$ 
for which the zeros $x_{4k+1}<x_{4k+3}$ are real and distinct. 

In the same way one can justify the disposition of the other points 
of the form $qx_j$ w.r.t. the points $x_j$.
\end{proof}

\begin{figure}[htbp]
\centerline{\hbox{\includegraphics[scale=0.7]{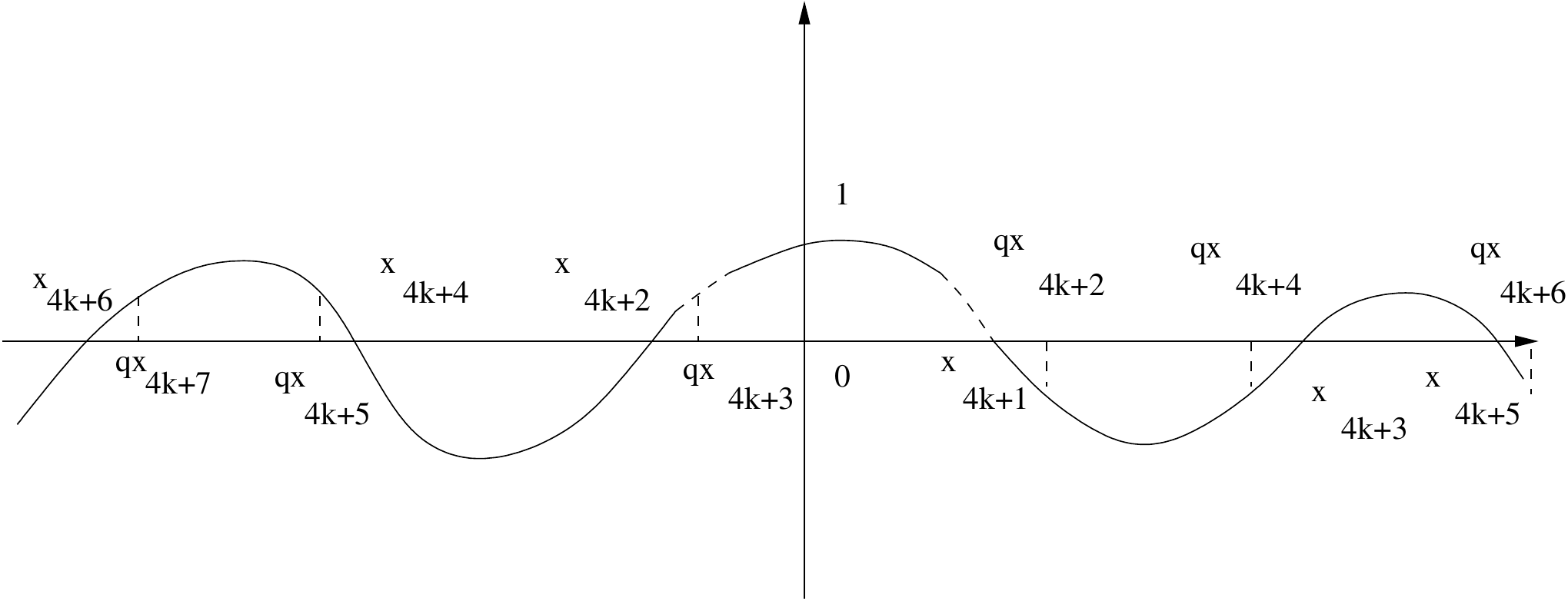}}}
    \caption{The real zeros of $\theta$ for $q\in (-1,0)$. }
\label{thetaqneg}
\end{figure}

\begin{prop}\label{firstzero}
The function $\theta (q,.)$ with $q\in (-1,0)$ has a zero in the interval 
$(0,-1/q)$. More precisely, one has $\theta (q,-1/q)<0$.
\end{prop}

\begin{proof}
Setting as above $v=-q$ one gets 

$$\theta (-v,1/v)=-v\Phi (v)+v^3\Psi (v)~~,
~~{\rm where}~~
\Phi (v)=\sum _{j=0}^{\infty}(-1)^jv^{2j^2+3j}~~,~~
\Psi (v)=\sum _{j=0}^{\infty}(-1)^jv^{2j^2+5j}~.$$
Further we use some results of \cite{Ko2}. Consider the functions 
$\varphi _k(\tau ):=\sum _{j=0}^{\infty}(-1)^j\tau ^{kj+j(j-1)/2}=
\theta (\tau ,-\tau ^{k-1})$ 
and $\xi _k(\tau ):=1/(1+\tau ^k)$, 
$\tau \in [0,1)$, $k>0$.

\begin{prop}\label{phixi}
(1) The functions $\varphi _k$ are real analytic on $[0,1)$; when $k>0$, then 
$\varphi _k<\xi _k$; when $k>1$, then $\varphi _k>\xi _{k-1}$; one has 
$\lim _{\tau \rightarrow 1^-}\varphi _k(\tau )=1/2$. 

(2) Consider the function $\varphi _k$ as a function of the two variables 
$(k,\tau )$. One has $\partial \varphi _k/\partial k>0$ for $k>0$. 

(3) For $q\in (0,1)$, $x\in (-q^{-1},\infty )$ one has 
$\partial \theta /\partial x>0$.

(4) For $q\in (0,1)$, $x\in (-q^{-1/2},\infty )$ one has $\theta >1/2$.
\end{prop}

Before proving Proposition~\ref{phixi} 
we finish the proof of Proposition~\ref{firstzero}. Part (2) of 
Proposition~\ref{phixi} implies 
$\varphi _{3/4}(v^4)>\varphi _{1/4}(v^4)$. 
As $-v\Phi (v)=\varphi _{1/4}(v^4)-1$ and 
$-v^3\Psi (v)=\varphi _{3/4}(v^4)-1$, this means that 
$\theta (-v,1/v)<0$, i.e. 
$\theta (q,-1/q)<0$. 
\end{proof}

\begin{proof}[Proof of Proposition~\ref{phixi}:] 

Part (1) of the proposition 
is proved in \cite{Ko2}.  

To prove part (2) observe that 

$$\begin{array}{rcl}\partial \varphi _k/\partial k&=&
(-\log \tau )\sum _{j=1}^{\infty}(-1)^{j-1}j\tau ^{kj+j(j-1)/2}\\ \\ 
&=&(-\log \tau )\tau ^k\sum _{j=1}^{\infty}(\varphi _{k+j}-\tau ^{k+j}\varphi _{k+j+1}) 
\end{array}$$
As $\varphi _{k+j}-\tau ^{k+j}\varphi _{k+j+1}=2\varphi _{k+j}-1$ and 
(see part (1) of the proposition) as 
$\varphi _{k+j}(\tau )>\xi _{k+j-1}(\tau )\geq 1/2$, each 
difference $\varphi _{k+j}-\tau ^{k+j}\varphi _{k+j+1}$ is positive on $[0,1)$. 
The factors $-\log \tau$ and $\tau ^k$ are also positive.

For $x\in (-q^{-1},0]$ part (3) follows from part (2) and from 
$\varphi _k(\tau )=\theta (\tau ,-\tau ^{k-1})$. Indeed, 
one can represent $x$ as $-\tau ^{k-1}$ for some $k>0$; for fixed $\tau$, 
as $x$ increases, $k$ also increases. One has 

$$0<\partial \varphi _k/\partial k=(-\log \tau)
\partial \theta /\partial x|_{x=-\tau ^{k-1}}~~{\rm and}~~-\log \tau >0~.$$
For $x>0$ part (3) results 
from all coefficients of $\theta (v,x)$ 
considered as a series in $x$ being positive.  

For $x\in (-q^{-1/2},0]$ part (4) follows from part (3). 
Indeed, consider the function 
$\psi :=1+2\sum _{j=1}^{\infty}(-1)^j\tau ^{j^2}$, $\tau \in [0,1)$. 
This function 
is positive valued, decreasing and tending to $0$ as $\tau \rightarrow 1^-$, 
see \cite{Ko1}. As $0<\psi (\tau ^{1/2})/2=\varphi _{1/2}(\tau )-1/2$, for 
$k\geq 1/2$ part (2) of Proposition~\ref{phixi} implies 

$$\theta (\tau ,-\tau ^{k-1})=
\varphi _{k}(\tau )\geq \varphi _{1/2}(\tau )>1/2~.$$ 
For $x=0$ one has $\theta =1$ hence for $x>0$ part (4) follows from part (2). 
\end{proof}

The following lemma follows immediately from the result of 
V.~Katsnelson cited in part (4) of Remarks~\ref{remsKa} and from 
Proposition~\ref{firstzero}:

\begin{lm}\label{absenceofzeros}
For any $\varepsilon >0$ sufficiently small there exists $\delta >0$ such that 
for $q\in (-1,-1+\delta ]$ the function $\theta (q,.)$ has a single real 
zero in the interval 
$(-e^{\pi /2}+\varepsilon ,e^{\pi /2}-\varepsilon )$. 
This zero is simple and positive.
\end{lm}

\subsection{How do the zeros of $\theta$ coalesce?}

Further we describe the way multiple zeros are formed when $q$ decreases 
from $0$ to $-1$.

\begin{defi}
{\rm We say that the phenomenon A happens before the phenomenon 
B if A takes place for $q=q_1$, B takes place for $q=q_2$ and 
$-1<q_2<q_1<0$. By phenomena we mean that certain zeros of $\theta$ or 
another function coalesce.}
\end{defi} 

\begin{nota}
{\rm We denote by $x_j\prec x_k$ the following statement: 
{\em The zeros $x_k$ and $x_{k+2}$ of $\theta$ can coalesce only after  
$x_j$ and $x_{j+2}$ have coalesced.}}
\end{nota}

\begin{lm}\label{prec}
One has $x_{4k+2}\prec x_{4k+3}$, $x_{4k+2}\prec x_{4k+1}$, 
$x_{4k+3}\prec x_{4k+4}$ and 
$x_{4k+5}\prec x_{4k+4}$, $k\in \mathbb{N}\cup 0$.
\end{lm}     

\begin{proof}
The statements follow respectively from 

$$\begin{array}{lcll}
qx_{4k+5}<x_{4k+4}<x_{4k+2}<qx_{4k+3}&,&x_{4k+1}<qx_{4k+2}<qx_{4k+4}<x_{4k+3}&,\\ \\ 
qx_{4k+4}<x_{4k+3}<x_{4k+5}<qx_{4k+6}&~~{\rm and}~~&x_{4k+6}<qx_{4k+7}<qx_{4k+5}<x_{4k+4}&. 
\end{array}$$
\end{proof}

\begin{lm}\label{prec24}
One has $x_{4k+2}\prec x_{4k+6}$, $k\in \mathbb{N}\cup 0$.


\end{lm}     

\begin{proof}
Indeed, equation (\ref{FE}) implies the following one:

\begin{equation}\label{fourfold}
\theta (q,x)=1+qx+q^3x^2+q^6x^3+q^{10}x^4\theta (q,q^4x)~.
\end{equation}
For $x=x_{4k+2}/q^4$ one obtains 

$$\theta (q,x_{4k+2}/q^4)=
1+x_{4k+2}/q^3+x_{4k+2}^2/q^5+x_{4k+2}^3/q^6=(1+x_{4k+2}^2/q^5)+(x_{4k+2}/q^3
+x_{4k+2}^3/q^6)~.$$
Each of the two sums is negative due to $q\in (-1,0)$, $x_{4k+2}<-1$ 
(see part (2) of Remarks~\ref{zerosoftheta}). 
For small values of $|q|$ one has $x_{4k+2}/q^4\in (x_{4k+8},x_{4k+6})$ 
because $x_j=-q^{-j}(1+O(q))$ and $\theta (q,x_{4k+2}/q^4)<0$. By continuity 
this holds true for all $q\in (-1,0)$ for which the zeros 
$x_{4k+8}$, $x_{4k+6}$, $x_{4k+4}$ 
and $x_{4k+2}$ are real and distinct. Hence if $x_{4k+2}$ 
and $x_{4k+4}$ have not coalesced, then $x_{4k+6}$ 
and $x_{4k+8}$ are real and distinct. 
\end{proof}

\begin{rems}\label{irreversible}
{\rm (1) Recall that for $q\in (-1,0)$ we set $v:=-q$ and that 
equation (\ref{uq}) holds true.} 


{\rm (2) Equation (\ref{DE}) implies that the values of $\theta$ 
at its local maxima decrease and its values at local minima increase as $q$ 
decreases in $(-1,0)$. Indeed, 
at a critical point one has $\partial \theta /\partial x=0$, 
so $2q\partial \theta /\partial q=x^2\partial ^2\theta /\partial x^2$. 
At a minimum one has $\partial ^2\theta /\partial x^2\geq 0$, so 
$\partial \theta /\partial q\leq 0$ and the value of $\theta$ increases as 
$q$ decreases; similarly for a maximum. 

(3) In particular, this means that $\theta$ can only lose real zeros, 
but not acquire such as $q$ decreases on $(-1,0)$. Indeed, the zeros of 
$\theta$ depend continuously on $q$. If at some point of the real axis a new 
zero of even multiplicity appears, then it cannot be a maximum because 
the critical value must decrease and it cannot be a minimum 
because its value must increase. 

(4) To treat the cases of odd multiplicities 
of the zeros it suffices to differentiate both sides of 
equation (\ref{DE}) w.r.t. $x$. 
For example, a simple zero $x_0$ of $\theta$ cannot become a triple one because 

$$2q\partial /\partial q(\partial \theta /\partial x)=
2\partial \theta /\partial x+4x\partial ^2\theta /\partial x^2+
x^2\partial ^3\theta /\partial x^3$$
which means that as   
$\partial \theta /\partial x|_{x=x_0}=\partial ^2\theta /\partial x^2|_{x=x_0}=0$,  
then either  
$\partial ^3\theta /\partial x^3|_{x=x_0}>0$ hence 
in a neighbourhood of $x_0$ one has $\partial \theta /\partial x\geq 0$ 
and $\partial /\partial q(\partial \theta /\partial x)<0$, or 
$\partial ^3\theta /\partial x^3|_{x=x_0}<0$ hence 
$\partial \theta /\partial x\leq 0$
and $\partial /\partial q(\partial \theta /\partial x)>0$ (in a neighbourhood 
of $x_0$), so in both cases 
the triple zero bifurcates into a simple one and a complex pair as 
$q$ decreases. The case of a zero of multiplicity $2m+1$, $m\in \mathbb{N}$, 
is treated by analogy. 



(5) In equation (\ref{uq}) the first argument (i.e. $v^4$) of both 
functions $\theta (v^4,-x^2/v)$ and 
$\theta (v^4,-vx^2)$ is the same, so when one of them has a double 
zero, then they both have double zeros. If the double zero of the first one 
is at $x=a$, then the one of the second is at $x=a/v$. }
\end{rems}
 
\begin{prop}\label{coalesceeven}
For any $k\in \mathbb{N}\cup 0$ there exists $q^*_k\in (-1,0)$ such that 
for $q=q^*_k$ the zeros $x_{4k+2}$ and $x_{4k+4}$ coalesce.
\end{prop}

\begin{figure}[htbp]
\centerline{\hbox{\includegraphics[scale=0.7]{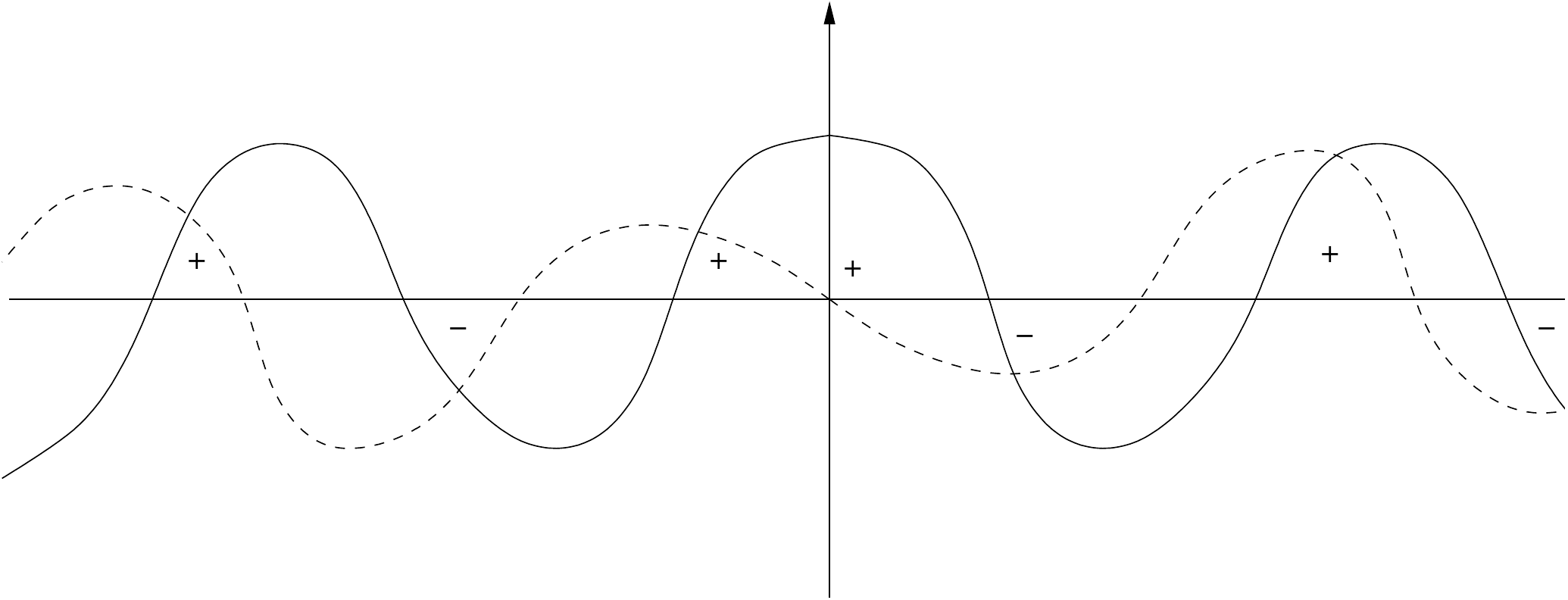}}}
    \caption{The graphs of $\psi _1$ (solid line) and 
$\psi _2$ (dashed line). }
\label{parthetatwographs}
\end{figure}

\begin{nota}\label{psinota}
{\rm We denote by $\psi _1$ and $\psi _2$ the functions 
$\theta (v^4,-x^2/v)$ and 
$-vx\theta (v^4,-vx^2)$. By $y_{\pm k}$ and $z_{\pm k}$ we 
denote their real zeros for $v^4\in (0,0.108]$, 
their moduli increasing with $k\in \mathbb{N}$, $y_{k}>0$, 
$y_{-k}<0$, $z_{k}>0$, $z_{-k}<0$. We set $z_0=0$.}
\end{nota}

\begin{proof}[Proof of Proposition~\ref{coalesceeven}:]
On Fig.~\ref{parthetatwographs} we show for $v^4\in (0,0.108]$ 
how the graphs of the functions 
$\psi _1$ and 
$\psi _2$ (drawn in solid and dashed line 
respectively) look like, the former being even and the latter odd, 
see part (1) of Remarks~\ref{irreversible}. The signs $+$ and $-$ indicate 
places, where it is certain that their sum $\theta (-v,x)$ 
is positive or negative. It is 
positive (resp. negative) if both functions are of this sign. It is positive 
near the origin because $\psi _1|_{x=0}>0$ while 
$\psi _2|_{x=0}=0$. 

For values of $v$ close to $0$ the zeros $y_{\pm k}$ of 
$\psi _1$ (resp. the zeros $z_{\pm k}$ of $\psi _2/x$) 
are close to the numbers $\pm v^{-(4k-1)/2}$ 
(resp. $\pm v^{-(4k+1)/2}$), 
$k\in \mathbb{N}$, see part (1) of Remarks~\ref{zerosoftheta}. From these 
remarks follows also that for small positive values of $v$ the 
zeros of $\theta (-v,x)$ are close to the numbers 
$-v^{-k}$, $k\in \mathbb{N}$. Hence on the negative half-axis $x$ 
one obtains the following arrangement of these numbers:

$$\begin{array}{cccccccccc}
\cdots &<-v^{-11/2}<&
-v^{-9/2}<&-v^{-4}&<-v^{-7/2}&<
-v^{-5/2}<&-v^{-2}&<-v^{-3/2}<0~.\\
&y_{-5}&z_{-3}&x_4&y_{-3}&z_{-1}&x_2&y_{-1}&\end{array}$$
Hence for $v^4\in (0,0.108]$ between a zone marked by a $+$ 
and one marked by a $-$ there is exactly one zero of $\theta (-v,.)$. 

As $v$ increases from $0$ to $1$, it takes countably-many values 
at each of which one of the functions 
$\psi _1$ and $\psi _2$ (in turn) has a double zero and when the value is 
passed, this double zero becomes a conjugate pair, 
see Theorem~\ref{known}. Hence on the negative real half-line 
the regions, where the corresponding function 
$\psi _1$ or $\psi _2$ is negative, disappear one by one starting 
from the right. As two consecutive changes of sign of $\theta (-v,.)$ 
are lost, $\theta (-v,.)$ has a couple of 
consecutive real negative zeros replaced by 
a complex conjugate pair. The quantity of these losses being countable implies 
the proposition.
\end{proof}

\begin{prop}\label{finmany}
Any interval of the form $[\gamma ,0]$, $\gamma \in (-1,0)$, contains at 
most finitely many spectral values of $q$.
\end{prop}

\begin{proof}
Indeed, if $\gamma \geq -0.108$, the interval contains 
no spectral value of $q$, 
see Lemma~\ref{lm108}; all zeros of $\theta$ are real and 
distinct and the graphs of 
the functions $\psi _{1,2}$ look as shown on Fig.~\ref{parthetatwographs}. 

Suppose that $\gamma <-0.108$. When $q$ decreases in $[\gamma ,0]$, for each of 
the functions this happens at most finitely many times that it has a double 
zero which then gives birth to a complex conjugate pair. 
This is always the zero which is closest to $0$, 
see part (2) of Theorem~\ref{known}. 

Hence the presentation of the graphs of $\psi _{1,2}$ 
changes only on some closed interval containing $0$, but outside it the zones 
marked by $+$ and $-$ continue to exist (but their borders change continuously) 
and the simple zeros of $\theta$ 
that are to be found between two such consecutive zones of opposite signs
are still to be found there. Besides, no new real zeros appear, see part (3) 
of Remarks~\ref{irreversible}. 

Therefore there exists $s_0\in \mathbb{N}$ such 
that when $q$ decreases from $0$ to $\gamma$, the zeros $x_{s_0}$, $x_{s_0+1}$, 
$x_{s_0+2}$, $\ldots$ remain simple and depend continuously on $q$. Hence 
only the rest of the zeros (i.e. $x_1$, $\ldots$, $x_{s_0-1}$) 
can participate in the bifurcations.
\end{proof}

\begin{prop}\label{onlydouble}
For $q\in (-1,0)$ the function $\theta (q,.)$ can have only simple 
and double real zeros. Positive double zeros are local maxima and negative 
double zeros are local minima.
\end{prop}

\begin{proof}
Equality (\ref{FE}) implies the following one:
$\theta (q,x/q^2)=1+x/q+(x^2/q)\theta (q,x)$. Set $x=x_{4k+2}$. Hence 
$x_{4k+2}<-1$, see part (2) of Remarks~\ref{zerosoftheta}. As 
$\theta (q,x_{4k+2})=0$ and $1+x_{4k+2}/q>0$, this implies 
$\theta (q,x_{4k+2}/q^2)>0$. In the same way $\theta (q,x_{4k+4}/q^2)>0$.

For $q$ close to $0$ the numbers $x_{4k+2}/q^2$ and $x_{4k+4}/q^2$ are close 
respectively to $x_{4k+4}$ and $x_{4k+6}$, 
see part (1) of Remarks~\ref{zerosoftheta}. 
Hence for such values of $q$ one has 

\begin{equation}\label{equation246}
x_{4k+6}<x_{4k+4}/q^2<x_{4k+2}/q^2<x_{4k+4}~.
\end{equation}
This string of inequalities holds true (by continuity) for  
$q$ belonging to any interval of the form $(a,0)$, $a\in (-1,0)$, 
for any $q$ of which 
the zeros $x_{4k+6}$, $x_{4k+4}$ and $x_{4k+2}$ are real and distinct. 

Equation (\ref{equation246}) implies that $x_{4k+6}$ and $x_{4k+4}$ cannot 
coalesce if $x_{4k+4}$ and $x_{4k+2}$ are real (but not necessarily distinct). 
Hence when for the 
first time negative zeros of $\theta$ coalesce, this happens with exactly two 
zeros, and the double zero is a local minimum of $\theta$. 

For positive zeros one obtains in the same way the string of inequalities   
$$x_{4k+5}<x_{4k+3}/q^2<x_{4k+5}/q^2<x_{4k+7}~.$$
Indeed, one has $1+x_{4k+3}/q<0$ because $x_{4k+3}>1$ (see Proposition~\ref{firstzero}) 
and $q\in (0,1)$. Hence $\theta (q,x_{4k+3}/q^2)<0$ and in the same 
way $\theta (q,x_{4k+5}/q^2)<0$. Thus  
$x_{4k+5}$ and $x_{4k+7}$ cannot 
coalesce if $x_{4k+3}$ and $x_{4k+5}$ are real (but not necessarily distinct). 
Hence when for the 
first time positive zeros of $\theta$ coalesce, this happens with exactly two 
zeros, and the double zero is a local maximum of $\theta$.

After a confluence of zeros takes place, one can give new indices to the 
remaining real zeros so that the indices of consecutive zeros differ by $2$ 
($x_{2s+2}<x_{2s}<0$ and $0<x_{2s+1}<x_{2s+3}$). After this for the next 
confluence the reasoning is the same.
\end{proof}


\begin{lm}\label{lm4k+3}
For $k\in \mathbb{N}$ sufficiently large one has $x_{4k+3}\prec x_{4k+6}$.
\end{lm}

\begin{proof}
Suppose that $x_{4k+3}\geq 3$. Then 

$$\theta (q,x_{4k+3}/q^3)=1+x_{4k+3}/q^2+x_{4k+3}^2/q^3+(x_{4k+3}^3/q^3)
\theta (q,x_{4k+3})=1+x_{4k+3}/q^2+x_{4k+3}^2/q^3~.$$  
For $x_{4k+3}\geq 3$ and $q\in (-1,0)$ the right-hand side is negative. In 
the same way $\theta (q,x_{4k+5}/q^3)<0$. We prove below that 

\begin{equation}\label{4string}
x_{4k+8}<x_{4k+5}/q^3<x_{4k+3}/q^3<x_{4k+6}~.
\end{equation}
Hence the zeros $x_{4k+8}$ and $x_{4k+6}$ cannot coalesce 
before $x_{4k+5}$ and $x_{4k+3}$ have coalesced. To prove the string of 
inequalities (\ref{4string}) observe that for $q$ close to $0$ the numbers 
$x_{4k+8}$ and $x_{4k+5}/q^3$ are close to one another ($x_{4k+3}/q^3$ and 
$x_{4k+6}$ as well, see part (1) of Remarks~\ref{zerosoftheta}) 
which implies (\ref{4string}). By continuity, as long as 
$x_{4k+3}\geq 3$ and $q\in (-1,0)$, the string of inequalities holds true 
also for $q$ not necessarily close to $0$. 

The result of V. Katsnelson (see part (4) of Remarks~\ref{remsKa}) implies 
that there exists $a\in (-1,0)$ such that for $q\in (-1,a]$ the function 
$\theta (q,x)$ has no zeros in $[-3,3]$ except the one which is simple and 
close to $1$ (see Proposition~\ref{firstzero}). 
Hence for $q\in (-1,a]$ the condition $x_{4k+3}\geq 3$ is 
fulfilled, if the zero $x_{4k+3}$ has been real and simple for 
$q\in (a,0)$. On the other hand for $q\in [a,0)$ only finitely many 
real zeros of $\theta$ have coalesced, and only finitely many 
have belonged to the interval $[-3,3]$ for some value of $q$, see 
Proposition~\ref{finmany}. 
Therefore there exists $k_0\in \mathbb{N}$ such that for $k\geq k_0$ one has 
$x_{4k+3}\geq 3$. 
\end{proof}

\subsection{Completion of the proof of Theorem~\protect\ref{new}}

Proposition~\ref{onlydouble} and Remarks~\ref{irreversible}
show that $\theta (q,.)$ has no real zero 
of multiplicity higher that $2$.  
Lemma~\ref{prec24} implies the string of inequalities 
$-1<\cdots <\bar{q}_{2l+2}<\bar{q}_{2l}<\cdots <0$. For $k$ sufficiently 
large one has $-1<\cdots <\bar{q}_{k+1}<\bar{q}_k<\cdots <0$. This 
follows from Lemmas~\ref{prec} and~\ref{lm4k+3}. It results from 
Proposition~\ref{finmany} and from the above inequalities 
that the set of spectral values has $-1$ as unique accumulation point. 
This proves part (3) of the theorem. 
 
Part (2) of the theorem results from Proposition~\ref{onlydouble}.

Part (4) follows from Remarks~\ref{irreversible}. These remarks show that 
real zeros can be only lost and that no new real zeros are born.

\section{Proof of Theorem~\protect\ref{astm}
\protect\label{secastm}}

Recall  that  
$\psi _1=\theta (v^4,-x^2/v)$ and 
$\psi _2=-vx\theta (v^4,-vx^2)$, see Notation~\ref{psinota}. 
Recall that the spectral values $\tilde{q}_k$ 
of $q$ for $\theta (q,x)$, $q\in (0,1)$ satisfy the asymptotic relation 
$\tilde{q}_k=1-\pi /2k+o(1/k)$. Hence the values of $v$ 
for which the function $\theta (v^4,x)$ has a double zero are of the 
form 

$$\tilde{v}_k=(\tilde{q}_k)^{1/4}=1-\pi /8k+o(1/k)$$ 
and the functions $\psi _{1,2}$ 
have double zeros for $v=\tilde{v}_k$. 

Consider three consecutive values of $k$ the first of which is odd -- 
$k_0$, $k_0+1$ and $k_0+2$. Set  $v:=\tilde{v}_{k_0}$. 
Denote by $a<b<0$ the double 
negative zeros of the functions $\psi _{1,2}|_{v=\tilde{v}_{k_0}}$. 
These zeros are local minima and on 
the whole interval $[a,b]$ one has $\theta >0$. The values of $\theta$ at 
local minima increase 
(see part (2) of Remarks~\ref{irreversible}), 
therefore the double zero of $\theta (\bar{q}_{k_0},.)$ 
is obtained for some  
$|q|<|\tilde{v}_{k_0}|$, 
i.e. before the functions $\psi _{1,2}|_{v=\tilde{v}_{k_0}}$ have double zeros. 
This follows from 
equality (\ref{uq}) in which both summands in the right-hand side have local 
minima (recall that as $k_0$ is odd, 
the double zero of $\theta$ is negative, so $x<0$ in equality (\ref{uq})). 
Hence 

\begin{equation}\label{eq11}
|\bar{q}_{k_0}|<|\tilde{v}_{k_0}|=1-\pi /8k_0+o(1/k_0)~.
\end{equation} 
In the same way 

\begin{equation}\label{eq22}
|\bar{q}_{k_0+2}|<|\tilde{v}_{k_0+2}|=1-\pi /8(k_0+2)+o(1/(k_0+2))~.
\end{equation}
In the case of $k_0+1$ the function $\psi _{1}|_{v=\tilde{v}_{k_0+1}}$ 
has a local minimum while 
$\psi _{2}|_{v=\tilde{v}_{k_0+1}}$ has a local maximum (because $k_0+1$ is even, 
the double zero of $\theta$ is positive, so $x>0$ in equality (\ref{uq})
$x>0$). As $\theta$ has a local maximum 
and as the values of $\theta$ at local maxima decrease 
(see part (2) of Remarks~\ref{irreversible}), 
the double zero of 
$\theta (\bar{q}_{k_0+1},.)$ is obtained for some  
$|q|>|\tilde{v}_{k_0+1}|$, i.e. after the functions 
$\psi _{1,2}|_{v=\tilde{v}_{k_0+1}}$ have double zeros. Therefore 

\begin{equation}\label{eq33}
|\bar{q}_{k_0+1}|>|\tilde{v}_{k_0+1}|=1-\pi /8(k_0+1)+o(1/(k_0+1))~. 
\end{equation}
When $k_0$ is sufficiently large one has 
$|\bar{q}_{k_0}|<|\bar{q}_{k_0+1}|<|\bar{q}_{k_0+2}|$ (this follows from part (3) 
of Theorem~\ref{new}). Using 
equations (\ref{eq22}) and (\ref{eq33}) one gets

$$1-\pi /8(k_0+1)+o(1/(k_0+1))<|\bar{q}_{k_0+1}|<|\bar{q}_{k_0+2}|<
1-\pi /8(k_0+2)+o(1/(k_0+2))~.$$
Hence $|\bar{q}_{k_0+1}|=1-\pi /8(k_0+1)+o(1/(k_0+1))$ and 
$|\bar{q}_{k_0+2}|=1-\pi /8(k_0+2)+o(1/(k_0+2))$. 
This implies the statement of Theorem~\ref{astm}.

\end{document}